\documentclass[12pt]{amsart}

\usepackage{amsmath,amssymb,amsfonts}
\usepackage{xcolor}
\usepackage{hyperref}

\usepackage{amsrefs}
\hypersetup{
  pdftitle={Scalar Curvature Flexibility in the Riemannian Burnett Compactness Class},
  pdfauthor={Jingbo Wan}
}

\newtheorem{theorem}{Theorem}[section]
\newtheorem{proposition}[theorem]{Proposition}
\newtheorem{lemma}[theorem]{Lemma}
\newtheorem{corollary}[theorem]{Corollary}
\theoremstyle{definition}
\newtheorem{definition}[theorem]{Definition}
\theoremstyle{remark}
\newtheorem{remark}[theorem]{Remark}

\newcommand{\Scal}{\operatorname{Scal}}
\newcommand{\Int}{\operatorname{Int}}
\newcommand{\TT}{\mathcal{T}}
\newcommand{\MM}{\mathcal{M}}
\newcommand{\BB}{\mathcal{B}}
\newcommand{\R}{\mathbb{R}}
\newcommand\pr{\partial}
\newcommand{\dd}{\,d}
\newcommand{\Met}{\operatorname{Met}}
\newcommand{\Spec}{\operatorname{Spec}}
\newcommand{\supp}{\operatorname{supp}}

\title[Scalar Curvature Flexibility]{Scalar Curvature Flexibility in the Riemannian Burnett Compactness Class}

\author{Jingbo Wan}
\address[Jingbo Wan]{Laboratoire Jacques-Louis Lions, Sorbonne Universit\'e, 4 place Jussieu, 75005 Paris, France}
\email{jingbo.wan@sorbonne-universite.fr}
\date{}

\begin{document}

\begin{abstract}
Let $M$ be a connected smooth $n$-manifold without boundary, where $n\geq3$, and let $\kappa\in\R$, with $\kappa\leq0$ if $M$ is open. We prove that every smooth Riemannian metric $g_0$ with $\Scal_{g_0}\geq\kappa$ is a locally uniform limit of smooth Riemannian metrics $g_i$ with $\Scal_{g_i}=\kappa$ that are locally uniformly bounded in $W^{1,\infty}$. As a corollary, combining this with Gromov's $C^0$-stability theorem, we obtain the perhaps surprising identity
\[
  \overline{\{g:\Scal_g=\kappa\}}^{\,C^{0,\alpha}_{\mathrm{loc}}}
  =\{g:\Scal_g\geq\kappa\}, \quad \forall \alpha \in(0,1).
\]
The restriction $\alpha<1$ is sharp. At $\kappa=0$, this proves and strengthens the Riemannian reverse-Burnett conjecture of Huneau and Luk.
\end{abstract}
\maketitle

\section{Introduction}

Burnett's two conjectures concern high-frequency limits of Lorentzian vacuum spacetimes. Their first-order compactness assumptions are local uniform convergence of the spacetime metrics and local uniform boundedness of their first derivatives \cites{Burnett,HuneauLuk}. Consequently, the first derivatives converge weak-* in $L^\infty_{\mathrm{loc}}$ to those of the limiting metric. We study the corresponding compactness regime for Riemannian metrics. It is natural for scalar curvature because, after multiplication by the volume density, scalar curvature has the local form
\[
  \sqrt{\det g}\,\Scal_g=\pr_kV^k(g,\pr g)+F(g,\pr g),
\]
where $V$ and $F$ depend smoothly on $g$ and are respectively linear and quadratic in $\pr g$; see, for example, \cite{LeeLeFloch}*{Section~2.1}. The divergence term is continuous in the sense of distributions under this convergence, whereas the quadratic term need not be. We determine the resulting relaxation of the equation $\Scal_g=\kappa$.

Let $M$ be a connected smooth manifold without boundary, and let $\Met(M)$ denote the space of smooth Riemannian metrics on $M$. For $\kappa\in\R$, set
\begin{align*}
  \MM_\kappa(M)&:=\{g\in\Met(M):\Scal_g=\kappa\},\\
  \MM_{\geq\kappa}(M)&:=\{g\in\Met(M):\Scal_g\geq\kappa\}.
\end{align*}
Throughout, all manifolds are assumed smooth, connected, and without boundary. A closed manifold is compact, while an open manifold is noncompact.

\begin{definition}\label{def:burnett}
Fix a smooth background $\overline g\in\Met(M)$, with Levi-Civita connection $\overline\nabla$. We say that smooth metrics $g_i$ converge to a smooth metric $g$ in the local Riemannian Burnett compactness class if, for every $K\Subset M$,
\[
  g_i\longrightarrow g\quad\text{in }C^0(K),\qquad
  \sup_i\|\overline\nabla g_i\|_{L^\infty(K)}<\infty.
\]
\end{definition}

This definition is independent of $\overline g$, since the difference of two smooth connections is a zeroth-order operator. Moreover, Lemma~\ref{lem:holder-compactness} gives
\[
  \overline\nabla g_i\stackrel{*}{\rightharpoonup}\overline\nabla g
  \quad\text{in }L^\infty_{\mathrm{loc}}.
\]

All local norms below may be taken with respect to $\overline g$. By convergence in the strong $C^{0,1}_{\mathrm{loc}}$-topology, we mean $C^0_{\mathrm{loc}}$-convergence together with strong $L^\infty_{\mathrm{loc}}$-convergence of the first derivatives.

\begin{theorem}\label{thm:main}
Let $M^n$ be connected, $n\geq3$, and let $\kappa\in\R$. If $M$ is closed or $\kappa\leq0$, every $g_0\in\MM_{\geq\kappa}(M)$ is the limit in the local Riemannian Burnett compactness class of a sequence $\widehat g_i\in\MM_\kappa(M)$. In particular, for every $K\Subset M$, the sequence may be chosen so that
\begin{equation}\label{eq:main-uniform-first-derivative}
  \sup_i\|\widehat g_i\|_{W^{1,\infty}(K;\overline g)}\leq C_K.
\end{equation}
If $\kappa<0$ and $g_0$ is complete, the metrics $\widehat g_i$ may be chosen complete.

This flexibility is sharp: if $g_i\in\MM_\kappa(M)$ converges to a smooth metric $g$ in $C^0_{\mathrm{loc}}$ and $\overline\nabla g_i\to\overline\nabla g$ strongly in $L^2_{\mathrm{loc}}$, then $g\in\MM_\kappa(M)$. In particular,
\begin{equation}\label{eq:lipschitz-closure}
  \overline{\MM_\kappa(M)}^{\,C^{0,1}_{\mathrm{loc}}}=\MM_\kappa(M).
\end{equation}
\end{theorem}

The theorem has the following consequence for Gromov's $C^0$ program. Gromov asked which consequences of scalar-curvature inequalities remain visible after the derivatives of the metric have been forgotten \cite{GromovLectures}*{Section~3.1}. His $C^0$-stability theorem gives \cite{Gromov}
\begin{equation*}
  \overline{\MM_\kappa(M)}^{\,C^0_{\mathrm{loc}}}\subseteq\MM_{\geq\kappa}(M),
\end{equation*}
and Bamler later gave a Ricci-flow proof \cite{Bamler}. This gives the forward inclusion below, while Theorem~\ref{thm:main} and Lemma~\ref{lem:holder-compactness} give the sharp reverse inclusion for every H\"older exponent $\alpha<1$.

\begin{corollary}\label{cor:closure}
If $M$ is closed or $\kappa\leq0$, then, for every $\alpha\in(0,1)$,
\begin{equation*}
  \overline{\MM_\kappa(M)}^{\,C^{0,\alpha}_{\mathrm{loc}}}=\MM_{\geq\kappa}(M),
\end{equation*}
and the same equality holds in the $C^0_{\mathrm{loc}}$-topology. When $M$ is closed, the topologies may be taken globally.
\end{corollary}

Thus the full relaxation permitted by Gromov's theorem already occurs under local $L^\infty$ control of the first derivatives. The scalar-curvature defect is carried by their failure to converge strongly.

We next explain the relation with General Relativity. The forward Burnett conjecture asserts that a high-frequency limit of vacuum spacetimes solves the Einstein--massless Vlasov system. The reverse conjecture asks whether suitable solutions of the Einstein--massless Vlasov system arise as high-frequency limits of vacuum spacetimes. The massless Vlasov stress-energy tensor $T$ satisfies the weak energy condition, namely $T(X,X)\geq0$ for every timelike vector $X$.

Huneau and Luk proved the forward conjecture under $\mathbb U(1)$ symmetry in elliptic gauge \cite{HuneauLukTrilinear} and, without symmetry, in generalized wave coordinates under quantitative high-frequency assumptions \cite{HuneauLukWave}. Guerra and Teixeira da Costa strengthened the elliptic-gauge result by removing the higher-derivative assumptions \cite{GuerraTeixeira}. In the reverse direction, Huneau and Luk constructed vacuum approximations of generic local-in-time small-data solutions to the Einstein--multiple-null-dust system under polarized $\mathbb U(1)$ symmetry \cite{HuneauLukBackreaction}, and subsequently treated suitable small, localized Einstein--massless Vlasov solutions under $\mathbb U(1)$ symmetry in elliptic gauge \cite{HuneauLukVlasov}. Without symmetry, Touati proved a reverse result for generic sufficiently regular solutions with a finite but arbitrary number of null-dust families \cite{TouatiReverse}.

At the level of initial data, a vacuum data set $(g,k)$ satisfies the constraint equations
\begin{equation*}\Scal_g-|k|g^2+(\operatorname{tr}gk)^2=0,\qquad \operatorname{div}gk-d(\operatorname{tr}gk)=0.\end{equation*}
The corresponding reverse problem for the constraint equations was studied by Touati, who constructed high-frequency vacuum initial data $(g_\lambda,k_\lambda)$ close to null-dust data on $\R^3$ \cite{TouatiConstraints}. His construction uses a coupled high-frequency expansion for both components of the initial data, designed to match a spacetime geometric-optics ansatz; imposing $k_\lambda\equiv0$ is therefore not a direct specialization of that construction. When $k\equiv0$, the vacuum constraint reduces to $\Scal_g=0$, whereas the time-symmetric weak energy condition reduces to $\Scal_g\geq0$. Huneau and Luk therefore asked whether every smooth metric of nonnegative scalar curvature is a $C^0_{\mathrm{loc}}$-limit of smooth scalar-flat metrics \cite{HuneauLuk}*{Conjecture~8.10}. This may be viewed as the time-symmetric Riemannian analogue of the reverse Burnett conjecture. Taking $\kappa=0$ in Theorem~\ref{thm:main} gives the following strengthening in Burnett's first-order compactness class.

\begin{corollary}\label{cor:huneau-luk}
Every smooth metric of nonnegative scalar curvature on a connected manifold without boundary of dimension $n\geq3$ is a limit in the local Riemannian Burnett compactness class of smooth scalar-flat metrics.
\end{corollary}

We now describe the proof. The prescribed scalar-curvature problem has a classical flexible side, including the work of Kazdan and Warner \cite{KazdanWarner} and the local deformation theory of Fischer and Marsden \cite{FischerMarsden}. In the $C^0$ setting used here, Lohkamp proved that an arbitrarily small perturbation can lower scalar curvature by a prescribed positive function up to an arbitrarily small error \cite{Lohkamp}. Aliouane, Rifford, and Theilli\`ere obtained a new proof by mixed convex integration \cite{ART}. We use their two-scale corrugation and scalar-curvature cancellation to construct the compactly supported parameterized deformation in Section~\ref{sec:path}, together with the uniform first-derivative bounds needed here.

Given a metric $g$ and a positive function $k$, Proposition~\ref{prop:global-path} constructs metrics $(g_s)_{s\in[0,1]}$ satisfying
\[
  \Scal_{g_s}\approx\Scal_g-s^2k
\]
while remaining arbitrarily close to $g$ in $C^0$ and locally bounded in $W^{1,\infty}$. Its endpoint is the almost-prescription theorem.

Almost prescription does not by itself produce a metric in $\MM_\kappa(M)$. We therefore make a global conformal change $\widehat g=u^{\frac{4}{n-2}}h$ that removes the remaining scalar-curvature error. A small error need not admit a small conformal correction because the linearized conformal operator may have a kernel or lack a uniform inverse. Once $h_i$ is locally bounded in $W^{1,\infty}$ and $u_i\to1$ locally uniformly, the conformal equation and the interior $W^{2,p}$ estimate give $u_i\to1$ in $C^1_{\mathrm{loc}}$; see Lemma \ref{lem:conformal-regularity}. Thus the correction remains in the local Riemannian Burnett compactness class.

Sections~\ref{sec:closed-flat} and~\ref{sec:open-flat} treat $\kappa=0$. On a closed manifold, we use the continuity of the first eigenvalue of the conformal Laplacian along the scalar-curvature decreasing path. At a suitable parameter, this eigenvalue vanishes, and its positive first eigenfunction gives an exact scalar-flat conformal metric. On an open manifold, a first eigenvalue need not exist. We first make the metric parabolic without changing it near a prescribed compact set. Parabolicity then provides the conformal corrections, and a compact exhaustion gives the result on the original manifold.

For $\kappa\neq0$, only the endpoint $s=1$ of Proposition~\ref{prop:global-path} is needed. When $\kappa<0$, constant sub- and supersolutions give the conformal correction; see Section~\ref{sec:negative-case}. When $\kappa>0$ and $M$ is closed, a small rescaling avoids resonance for the linearized conformal equation, after which a contraction argument gives the correction; see Section~\ref{sec:positive-closed}.

The case $\kappa>0$ on general open manifolds remains open. The local almost-prescription remains available, but we do not know how to remove its error by a global conformal correction tending to the identity in $C^1_{\mathrm{loc}}$.

\medskip

{\it Acknowledgments.} The author would like to thank Otis Chodosh and Jonathan Luk for bringing this problem to his attention and for their interest. He thanks Jonathan Luk in particular for pointing out the significance of the uniform first-derivative bounds in relation to Burnett's compactness regime.  He also thanks Sung-Jin Oh for helpful and encouraging discussions, Man-Chun Lee for useful discussions about the proof, and Mu-Tao Wang for his continued support and encouragement. The author is supported by ERC-2023 AdG 101141855 BLaHST.

\section{First-order compactness}\label{sec:first-order}

We use the convention $\Delta_g=\operatorname{div}_g\nabla_g$, so that $-\Delta_g$ is nonnegative. Set
\[
  c_n:=\frac{4(n-1)}{n-2},\qquad L_g:=-c_n\Delta_g+\Scal_g.
\]
For every positive function $u$, the conformal metric $\widehat g=u^{4/(n-2)}g$ satisfies
\begin{equation}\label{eq:conformal-formula}
  \Scal_{\widehat g}=u^{-\frac{n+2}{n-2}}L_gu.
\end{equation}

We record two simple consequences of a local first-derivative bound.

\begin{lemma}\label{lem:holder-compactness}
Let $T_i\to T$ be smooth tensors converging in $C^0_{\mathrm{loc}}$, and
\[
  \sup_i\|\overline\nabla T_i\|_{L^\infty(K)}<\infty
  \quad\text{for every }K\Subset M.
\]
Then $T_i\to T$ in $C^{0,\alpha}_{\mathrm{loc}}$ for every $\alpha\in(0,1)$, and $\overline\nabla T_i\rightharpoonup^*\overline\nabla T$ in $L^\infty_{\mathrm{loc}}$.
\end{lemma}

\begin{proof}
Fix $K\Subset K'\Subset M$. The interpolation inequality in a finite collection of coordinate charts gives
\[
  \|T_i-T\|_{C^{0,\alpha}(K)}\leq C_{K,K'}\|T_i-T\|_{C^0(K')}^{1-\alpha}\left(\|T_i-T\|_{C^0(K')}+\|\overline\nabla(T_i-T)\|_{C^0(K')}\right)^\alpha,
\]
which tends to zero. For every smooth compactly supported test tensor $\Phi$, integration by parts gives
\[
  \int_M\langle\overline\nabla T_i,\Phi\rangle_{\overline g}\dd\mu_{\overline g}
  =\int_M\langle T_i,\overline\nabla^*\Phi\rangle_{\overline g}\dd\mu_{\overline g}
  \longrightarrow
  \int_M\langle T,\overline\nabla^*\Phi\rangle_{\overline g}\dd\mu_{\overline g}.
\]
Thus $\overline\nabla T_i\to\overline\nabla T$ distributionally. The uniform local $L^\infty$-bound and density in $L^1$ upgrade this to weak-* convergence of the complete sequence.
\end{proof}

\begin{lemma}\label{lem:conformal-regularity}
Let $h_i$ be smooth metrics converging to $h$ in $C^0_{\mathrm{loc}}$, with
\[
  \sup_i\|\overline\nabla h_i\|_{L^\infty(K)}<\infty
  \quad\text{for every }K\Subset M,
\]
and suppose that $\Scal_{h_i}\to\kappa$ in $C^0_{\mathrm{loc}}$. If positive smooth functions $u_i\to1$ in $C^0_{\mathrm{loc}}$ satisfy
\[
  L_{h_i}u_i=\kappa u_i^{\frac{n+2}{n-2}},
\]
then $u_i\to1$ in $W^{2,p}_{\mathrm{loc}}$ for every $1<p<\infty$, and hence in $C^1_{\mathrm{loc}}$. Moreover,
\[
  \sup_i\left\|u_i^{\frac{4}{n-2}}h_i\right\|_{W^{1,\infty}(K;\overline g)}<\infty
  \quad\text{for every }K\Subset M.
\]
\end{lemma}

\begin{proof}
Set $v_i:=u_i-1$. Equation $L_{h_i}u_i=\kappa u_i^{\frac{n+2}{n-2}}$ gives
\[
  \Delta_{h_i}v_i=\frac{1}{c_n}\left(\Scal_{h_i}u_i-\kappa u_i^{\frac{n+2}{n-2}}\right)=:f_i \to0 \quad \text{in } C^0_{\mathrm{loc}}.
\]
Fix $K\Subset K'$ in a coordinate chart. The operators $\Delta_{h_i}$ are uniformly elliptic on $K'$, with uniformly Lipschitz leading coefficients and uniformly bounded first-order coefficients. Hence the interior estimate \cite{GT}*{Theorem~9.11} gives
\[
  \|v_i\|_{W^{2,p}(K)}\leq C\left(\|v_i\|_{L^p(K')}+\|f_i\|_{L^p(K')}\right)\longrightarrow0.
\]
Taking $p>n$ gives $u_i\to1$ in $C^1_{\mathrm{loc}}$. Finally,
\[
  \overline\nabla\left(u_i^{\frac{4}{n-2}}h_i\right)
  =\frac{4}{n-2}u_i^{\frac{6-n}{n-2}}\dd u_i\otimes h_i
  +u_i^{\frac{4}{n-2}}\overline\nabla h_i,
\]
which is locally uniformly bounded.
\end{proof}

\begin{proof}[Proof of the sharpness assertion in Theorem~\ref{thm:main}]
We use the invariant divergence form of scalar curvature relative to $\overline g$; see, for example, \cite{LeeLeFloch}*{Section~2.1}. Set
\[
  C^k_{ij}(g):=\frac12g^{k\ell}\left(\overline\nabla_i g_{j\ell}+\overline\nabla_jg_{i\ell}-\overline\nabla_\ell g_{ij}\right),\qquad
  \vartheta_g:=\frac{\dd\mu_g}{\dd\mu_{\overline g}},
\]
so that $C(g)=\nabla^g-\overline\nabla$. Define
\begin{align*}
  V^k(g,\overline\nabla g)&:=\vartheta_g\left(g^{ij}C^k_{ij}(g)-g^{ik}C^j_{ij}(g)\right),\\
  F(g,\overline\nabla g)&:=\vartheta_gg^{ij}(\operatorname{Ric}_{\overline g})_{ij}-\overline\nabla_k(\vartheta_gg^{ij})C^k_{ij}(g)+\overline\nabla_j(\vartheta_gg^{ij})C^k_{ik}(g)\\
  &\quad+\vartheta_gg^{ij}\left(C^k_{ij}(g)C^\ell_{k\ell}(g)-C^\ell_{ik}(g)C^k_{j\ell}(g)\right).
\end{align*}
The curvature-difference formula gives
\begin{equation}\label{eq:scalar-density-divergence}
  \vartheta_g\Scal_g=\overline\nabla_kV^k(g,\overline\nabla g)+F(g,\overline\nabla g).
\end{equation}
Here $V$ is smooth in $g$ and linear in $\overline\nabla g$, while, apart from the smooth background-curvature term, $F$ is smooth in $g$ and quadratic in $\overline\nabla g$.

Set $\vartheta_i:=\vartheta_{g_i}$. Since $g_i\to g$ locally uniformly and $\overline\nabla g_i\to\overline\nabla g$ strongly in $L^2_{\mathrm{loc}}$, the preceding formulas give
\[
  \vartheta_i\longrightarrow\vartheta_g\quad\text{in }C^0_{\mathrm{loc}},\qquad
  V(g_i,\overline\nabla g_i)\longrightarrow V(g,\overline\nabla g)\quad\text{in }L^2_{\mathrm{loc}},
\]
and
\[
  F(g_i,\overline\nabla g_i)\longrightarrow F(g,\overline\nabla g)
  \quad\text{in }L^1_{\mathrm{loc}}.
\]
For every $\phi\in C_c^\infty(M)$, equation \eqref{eq:scalar-density-divergence} and $\Scal_{g_i}=\kappa$ give
\[
  -\int_MV^k(g_i,\overline\nabla g_i)\overline\nabla_k\phi\dd\mu_{\overline g}
  +\int_MF(g_i,\overline\nabla g_i)\phi\dd\mu_{\overline g}
  =\kappa\int_M\vartheta_i\phi\dd\mu_{\overline g}.
\]
Passing to the limit gives $\vartheta_g\Scal_g=\kappa\vartheta_g$ in the sense of distributions. Since $g$ is smooth and $\vartheta_g>0$, we obtain $\Scal_g=\kappa$. Finally, strong $C^{0,1}_{\mathrm{loc}}$-convergence implies strong $W^{1,2}_{\mathrm{loc}}$-convergence, which proves \eqref{eq:lipschitz-closure}.
\end{proof}

\section{Scalar curvature decreasing paths}\label{sec:path}

Convex integration originated in Nash's work on $C^1$ isometric embeddings \cite{Nash} and was developed by Gromov into the general framework of the $h$-principle \cite{GromovPDR}; for an introduction to these ideas, see the survey of De Lellis and Sz\'ekelyhidi \cite{DeLellisSzekelyhidi}. Standard convex integration relies on an ampleness condition. The scalar-curvature relation, however, is not ample because scalar curvature is affine in the second derivatives of the metric. The mixed convex integration of Aliouane, Rifford, and Theilli\`ere overcomes this obstruction by coupling oscillations at two scales: the first generates a signed quadratic term, and the second cancels the leading affine term \cite{ART}. In this section, we give the compactly supported, parameterized construction needed below and keep track of its first derivatives. The perturbations tend to zero in $C^0$ and remain bounded in $W^{1,\infty}_{\mathrm{loc}}$, as required by the Riemannian Burnett compactness class.

Let $U\Subset\R^n$, where $n\geq3$, and let $h$ be a smooth metric on $U$.  Applying Gram--Schmidt to $\{\pr_j\}_{j=1}^{n-1}$ gives an $h$-orthonormal frame $\{X^j\}_{j=1}^{n-1}$.  Let $X^n$ be the $h$-unit normal to $\operatorname{span}\{X^j\}_{j=1}^{n-1}$, oriented so that $\alpha:=dx^n(X^n)>0$. Write
\[
  X^n = \alpha\pr_n+\sum_{j=1}^{n-1}\beta_j\pr_j.
\]
Let $\{\theta^j\}_{j=1}^n$ be the dual coframe, so that $h=\sum_{i=1}^n(\theta^i)^2$. For a smooth map $F=(F^1,F^2):U\to\R^2$, define
\begin{equation}\label{eq:perturbed-metric}
  h_F := e^{2F^1}(\theta^1)^2 +e^{2F^2}(\theta^2)^2 +\sum_{i=3}^n(\theta^i)^2.
\end{equation}
We collect the $\pr_n$-free data of $F$ up to second order in
\[
  \TT(F)  := \left( F,\, (\pr_jF)_{j=1}^{n-1},\, (\pr_j\pr_kF)_{1\leq j\leq k\leq n-1} \right).
\]

We first record the local scalar curvature identity underlying the construction; see also \cite{ART}*{Lemma~3.4}. For readers' convenience, the computation is included to make the dependence on the background metric explicit.

\begin{lemma}\label{lem:local-identity}
There are smooth functions $\Psi,\Psi_1,\Psi_2$, depending on $(x,\TT(F))$, such that $\Psi(x,0)=0$ and
\begin{align}
  \Scal_{h_F}-\Scal_h ={}& -2\alpha^2\sum_{k=1}^2\pr_{nn}^2F^k -4\alpha\sum_{j=1}^{n-1}\sum_{k=1}^2 \beta_j\pr_{nj}^2F^k \notag\\
  &-2\alpha^2\left( (\pr_nF^1)^2+(\pr_nF^2)^2 +\pr_nF^1\pr_nF^2 \right)\notag\\
  &+\sum_{k=1}^2 \Psi_k\bigl(x,\TT(F)\bigr)\pr_nF^k +\Psi\bigl(x,\TT(F)\bigr).
  \label{eq:local-identity}
\end{align}
These functions depend smoothly on $h$ when $h$ varies in a smooth family of metrics.
\end{lemma}

\begin{proof}
Fix $x\in U$. Choose affine coordinates $y_1,\ldots,y_n$ centered at $x$ such that $\pr_{y_i}=X^i(x)$ at the origin. Write $\theta^q=\sum_{i=1}^n\theta_i^q\,dy_i$, and extend $F=(F^1,F^2)$ by setting $F^3=\cdots=F^n=0$. Then
\[
  \theta_i^q(x)=\delta_{iq},\qquad h_{ij}=\sum_q\theta_i^q\theta_j^q,\qquad (h_F)_{ij}=\sum_qe^{2F^q}\theta_i^q\theta_j^q.
\]

Consider the diagonal metric $\widehat h_{ij}=e^{2F^i}\delta_{ij}$. Since $\widehat h^{ij}=e^{-2F^i}\delta_{ij}$ and $\pr_k\widehat h_{ij}=2e^{2F^i}\delta_{ij}\pr_kF^i$, we have
\begin{align*}
  \widehat\Gamma^k_{ij} &= \sum_\ell\frac12\widehat h^{k\ell}\left(\pr_i\widehat h_{j\ell}+\pr_j\widehat h_{i\ell}-\pr_\ell\widehat h_{ij}\right)\\
  &=\delta_{jk}\pr_iF^k+\delta_{ik}\pr_jF^k-\delta_{ij}e^{2(F^i-F^k)}\pr_kF^i,\\
  \sum_{i,j,k}\widehat h^{ij}\pr_k\widehat\Gamma^k_{ij} &=\sum_{i,k}e^{-2F^i}\pr_k\left(2\delta_{ik}\pr_iF^i-e^{2(F^i-F^k)}\pr_kF^i\right),\\
  -\sum_{i,j,k}\widehat h^{ij}\pr_j\widehat\Gamma^k_{ik} &=-\sum_{i,k}e^{-2F^i}\pr_{ii}^2F^k.
\end{align*}
It follows that
\begin{align}
  \Scal_{\widehat h} ={}& \sum_{i,j,k}\widehat h^{ij}\left(\pr_k\widehat\Gamma^k_{ij}-\pr_j\widehat\Gamma^k_{ik}+\sum_\ell\widehat\Gamma^k_{ij}\widehat\Gamma^\ell_{k\ell}-\sum_\ell\widehat\Gamma^\ell_{ik}\widehat\Gamma^k_{j\ell}\right)\notag\\
  ={}&-2\sum_{a=1}^2\sum_{k\neq a}e^{-2F^k}\pr_{kk}^2F^a-2\sum_{a=1}^2\sum_{k\neq a}e^{-2F^k}(\pr_kF^a)^2 \label{eq:diagonal-scalar-path}\\
  &-2\sum_{k\notin\{1,2\}}\pr_kF^1\pr_kF^2+2e^{-2F^1}\pr_1F^1\pr_1F^2+2e^{-2F^2}\pr_2F^2\pr_2F^1. \notag
\end{align}

We now compare $h_F$ with $\widehat h$. At $x$, we have $(h_F)_{ij}=\widehat h_{ij}$ and $(h_F)^{ij}=\widehat h^{ij}=e^{-2F^i}\delta_{ij}$. Differentiating and using $\theta_i^q(x)=\delta_{iq}$ gives
\begin{align}
  \pr_r(h_F)_{ij}-\pr_r\widehat h_{ij} ={}&\sum_qe^{2F^q}\left(\pr_r\theta_i^q\,\delta_{jq}+\delta_{iq}\pr_r\theta_j^q\right), \label{eq:first-coframe-error}\\
  \pr_{rs}^2(h_F)_{ij}-\pr_{rs}^2\widehat h_{ij} ={}&2\sum_qe^{2F^q}\pr_rF^q\left(\pr_s\theta_i^q\,\delta_{jq}+\delta_{iq}\pr_s\theta_j^q\right)\notag\\
  &+2\sum_qe^{2F^q}\pr_sF^q\left(\pr_r\theta_i^q\,\delta_{jq}+\delta_{iq}\pr_r\theta_j^q\right)\notag\\
  &+\sum_qe^{2F^q}\left(\pr_{rs}^2\theta_i^q\,\delta_{jq}+\pr_r\theta_i^q\pr_s\theta_j^q+\pr_s\theta_i^q\pr_r\theta_j^q+\delta_{iq}\pr_{rs}^2\theta_j^q\right). \label{eq:second-coframe-error}
\end{align}
For any metric $g$,
\begin{align*}
  \pr_r\Gamma^k_{ij} ={}&\frac12\sum_\ell\pr_rg^{k\ell}\left(\pr_ig_{j\ell}+\pr_jg_{i\ell}-\pr_\ell g_{ij}\right)+\frac12\sum_\ell g^{k\ell}\left(\pr_{ri}^2g_{j\ell}+\pr_{rj}^2g_{i\ell}-\pr_{r\ell}^2g_{ij}\right),
\end{align*}
where $\pr_rg^{k\ell}=-\sum_{p,q}g^{kp}g^{\ell q}\pr_rg_{pq}$. Moreover,
\begin{align*}
  \Gamma^k_{ij}(h_F)-\widehat\Gamma^k_{ij} ={}&\frac12e^{-2F^k}\sum_qe^{2F^q}\Bigl(\pr_i\theta_j^q\,\delta_{kq}+\delta_{jq}\pr_i\theta_k^q+\pr_j\theta_i^q\,\delta_{kq}+\delta_{iq}\pr_j\theta_k^q\\
  &\hspace{42mm}-\pr_k\theta_i^q\,\delta_{jq}-\delta_{iq}\pr_k\theta_j^q\Bigr).
\end{align*}
Equations \eqref{eq:first-coframe-error} and \eqref{eq:second-coframe-error} show that the coframe errors contain no second derivative of $F$ and are at most linear in $\pr F$. Since $\widehat\Gamma^k_{ij}$ is linear in $\pr F$, the same is true, up to terms containing no derivative of $F$, for the differences of the differentiated and quadratic Christoffel terms. Hence $\Scal_{h_F}-\Scal_{\widehat h}$ contains no second derivative of $F$ and no term quadratic in $\pr F$. After subtracting $\Scal_h$, the derivative-free terms vanish when $F=0$. Thus all second-derivative and quadratic first-derivative terms in $\Scal_{h_F}-\Scal_h$ are exactly those in \eqref{eq:diagonal-scalar-path}.

It remains to rewrite the derivatives in the original coordinates. At $x$, we have $\pr_{y_n}=\alpha\pr_n+\sum_{j=1}^{n-1}\beta_j\pr_j$. Here $\alpha$ and $\beta_j$ denote their values at $x$ and are constant coefficients in the affine coordinate change. Therefore, for $a=1,2$,
\begin{align*}
  \pr_{y_n}F^a &=\alpha\pr_nF^a+\sum_{j=1}^{n-1}\beta_j\pr_jF^a,\\
  \pr_{y_ny_n}^2F^a &=\alpha^2\pr_{nn}^2F^a+2\alpha\sum_{j=1}^{n-1}\beta_j\pr_{nj}^2F^a+\sum_{j,k=1}^{n-1}\beta_j\beta_k\pr_{jk}^2F^a.
\end{align*}
Consequently,
\begin{align*}
  -2\sum_{a=1}^2\pr_{y_ny_n}^2F^a
  ={}-2\alpha^2\sum_{a=1}^2\pr_{nn}^2F^a-4\alpha\sum_{j=1}^{n-1}\sum_{a=1}^2\beta_j\pr_{nj}^2F^a
  -2\sum_{j,k=1}^{n-1}\sum_{a=1}^2\beta_j\beta_k\pr_{jk}^2F^a,
\end{align*}
and
\begin{align*}
  {}&(\pr_{y_n}F^1)^2+(\pr_{y_n}F^2)^2+\pr_{y_n}F^1\pr_{y_n}F^2\\
  ={}&\alpha^2\left((\pr_nF^1)^2+(\pr_nF^2)^2+\pr_nF^1\pr_nF^2\right)\\
  &+\alpha\pr_nF^1\sum_{j=1}^{n-1}\beta_j\left(2\pr_jF^1+\pr_jF^2\right)+\alpha\pr_nF^2\sum_{j=1}^{n-1}\beta_j\left(\pr_jF^1+2\pr_jF^2\right)\\
  &+\left(\sum_{j=1}^{n-1}\beta_j\pr_jF^1\right)^2+\left(\sum_{j=1}^{n-1}\beta_j\pr_jF^2\right)^2+\left(\sum_{j=1}^{n-1}\beta_j\pr_jF^1\right)\left(\sum_{k=1}^{n-1}\beta_k\pr_kF^2\right).
\end{align*}

For $q\leq n-1$, the vector $\pr_{y_q}$ is a linear combination of $\pr_1,\ldots,\pr_{n-1}$. Hence the remaining terms depend only on $\TT(F)$, except for terms linear in $\pr_nF^1,\pr_nF^2$. Collecting them gives \eqref{eq:local-identity}. Taking $F\equiv0$ gives $\Psi(x,0)=0$, and smooth dependence on $h$ follows from that of the frame.
\end{proof}

For a smooth $1$-periodic loop $\zeta:\R\to\R^2$, let $\bar\zeta:=\int_0^1\zeta(t)\dd t$, and denote by $\Int(\zeta)$ the unique $1$-periodic loop of mean zero satisfying
\[
  \pr_t\Int(\zeta)=\zeta-\bar\zeta.
\]
We write $\Int^2:=\Int\circ\Int$. Thus, if $\bar\zeta=0$, then $\pr_t\Int(\zeta)=\zeta$ and $\pr_{tt}^2\Int^2(\zeta)=\zeta$. The next lemma gives the construction in a compactly supported, parameterized form. It is the local counterpart of \cite{ART}*{Proposition~3.5} and is based on the two-scale corrugation in \cite{ART}*{Proposition~2.2}.

\begin{lemma}\label{lem:supported-deformation}
Let $(h_s)_{s\in[0,1]}$ be a smooth family of metrics on $U\Subset\R^n$, and let $a\in C^\infty([0,1]\times U)$ satisfy
\[
  \supp_x a:=\overline{\{x\in U:a(s,x)\neq0\text{ for some }s\in[0,1]\}}\Subset U.
\]
For every $\varepsilon,\eta>0$, there is a smooth family of metrics $(\widetilde h_s)_{s\in[0,1]}$ such that
\begin{align}
  \sup_{s\in[0,1]}
  \|\widetilde h_s-h_s\|_{C^0}
  &<\varepsilon,
  \label{eq:local-c0}\\
  \sup_{(s,x)\in[0,1]\times U}
  \left|
    \Scal_{\widetilde h_s}(x)
    -\Scal_{h_s}(x)
    +a(s,x)^2
  \right|
  &<\eta,
  \label{eq:local-scal}\\
  \widetilde h_s
  &=
  h_s
  \quad\text{on }U\setminus\supp_x a.
  \notag
\end{align}
If $a(0,\cdot)=0$, then $\widetilde h_0=h_0$. Moreover, if $\Lambda^{-1}\delta\leq h_s\leq\Lambda\delta$, $\sup_s\|h_s\|_{C^1(U)}\leq H$, and $\|a\|_{C^0([0,1]\times U)}\leq A$, then the output may be chosen so that
\begin{equation}\label{eq:local-uniform-first-derivative}
  \sup_{s\in[0,1]}\|\widetilde h_s\|_{W^{1,\infty}(U)}\leq C(U,\Lambda,H,A),
\end{equation}
\end{lemma}

\begin{proof}
Choose the frame $X^i$ and the dual coframe $\theta^i$ smoothly in $s$.  We use the notation $\alpha_s,\{\beta_{j,s}\}_{j=1}^{n-1}$ and $\Psi_s,\Psi_{1,s},\Psi_{2,s}$ for the corresponding functions in Lemma~\ref{lem:local-identity}.  On the spatial support of $a$, the function $\alpha_s(x)$ has a uniform positive lower bound.

Define $1$-periodic loops $\gamma_{s,x},\sigma_{s,x}:\R\to\R^2$ by
\[
  q_{s,x}(t) := \frac{a(s,x)}{\alpha_s(x)}\cos(2\pi t), \qquad \gamma_{s,x}(t) := \bigl(q_{s,x}(t),-q_{s,x}(t)\bigr).
\]
Set $(\sigma_{s,x})_2=0$, and define $(\sigma_{s,x})_1$ by
\begin{equation}\label{eq:delta-loop}
  (\sigma_{s,x})_1(t) := -q_{s,x}(t)^2 +\frac{1}{2\alpha_s(x)^2}\bigl[ \Psi_{1,s}(x,0)-\Psi_{2,s}(x,0) \bigr]q_{s,x}(t) +\frac{a(s,x)^2}{2\alpha_s(x)^2}.
\end{equation}
Since $\overline{q_{s,x}}=0$ and $\overline{q_{s,x}^2}=a(s,x)^2/(2\alpha_s(x)^2)$, we have $\bar\gamma_{s,x}=0$ and $\bar\sigma_{s,x}=0$. Both loops vanish wherever $a=0$.

For a large positive integer $N$, define
\begin{equation}\label{eq:convex-integration}
  F_{s,N}(x):=\frac1N\Int(\gamma_{s,x})(Nx_n)+\frac1{N^2}\Int^2(\sigma_{s,x})(Nx_n).
\end{equation}
All loop terms below are evaluated at $t=Nx_n$, and their derivatives with respect to $x$ are taken with $t$ fixed.
For $j,k\leq n-1$, differentiation gives
\begin{align*}
  \pr_jF_{s,N}&=\frac1N\pr_j\Int(\gamma_{s,x})+\frac1{N^2}\pr_j\Int^2(\sigma_{s,x}),\\
  \pr_{jk}^2F_{s,N}&=\frac1N\pr_{jk}^2\Int(\gamma_{s,x})+\frac1{N^2}\pr_{jk}^2\Int^2(\sigma_{s,x}).
\end{align*}
Differentiating with respect to $x_n$ gives, for $1\leq j\leq n-1$,
\begin{align*}
  \pr_nF_{s,N}
  ={}&\pr_t\Int(\gamma_{s,x})+\frac1N\left[\pr_n\Int(\gamma_{s,x})+\pr_t\Int^2(\sigma_{s,x})\right]+\frac1{N^2}\pr_n\Int^2(\sigma_{s,x})\\
  ={}&\gamma_{s,x}+\frac1N\left[\pr_n\Int(\gamma_{s,x})+\Int(\sigma_{s,x})\right]+\frac1{N^2}\pr_n\Int^2(\sigma_{s,x}),\\
  \pr_{nj}^2F_{s,N}
  ={}&\pr_{jt}^2\Int(\gamma_{s,x})+\frac1N\left[\pr_{nj}^2\Int(\gamma_{s,x})+\pr_{jt}^2\Int^2(\sigma_{s,x})\right]+\frac1{N^2}\pr_{nj}^2\Int^2(\sigma_{s,x})\\
  ={}&\pr_j\gamma_{s,x}+\frac1N\left[\pr_{nj}^2\Int(\gamma_{s,x})+\pr_j\Int(\sigma_{s,x})\right]+\frac1{N^2}\pr_{nj}^2\Int^2(\sigma_{s,x}),\\
  \pr_{nn}^2F_{s,N}
  ={}&N\pr_{tt}^2\Int(\gamma_{s,x})+2\pr_{nt}^2\Int(\gamma_{s,x})+\frac1N\pr_{nn}^2\Int(\gamma_{s,x})+\pr_{tt}^2\Int^2(\sigma_{s,x})\\
  &+\frac2N\pr_{nt}^2\Int^2(\sigma_{s,x})+\frac1{N^2}\pr_{nn}^2\Int^2(\sigma_{s,x})\\
  ={}&N\pr_t\gamma_{s,x}+2\pr_n\gamma_{s,x}+\sigma_{s,x}+\frac1N\left[\pr_{nn}^2\Int(\gamma_{s,x})+2\pr_n\Int(\sigma_{s,x})\right]\\
  &+\frac1{N^2}\pr_{nn}^2\Int^2(\sigma_{s,x}).
\end{align*}
Since $t\in\R/\mathbb Z$ and $\supp_x a$ is compact, the loops and all their spatial derivatives are uniformly bounded on $[0,1]\times\supp_x a\times\R/\mathbb Z$. Consequently,
\begin{align*}
  \TT(F_{s,N}) &= O(N^{-1}),\\
  \pr_nF_{s,N} &= \gamma_{s,x}(Nx_n)+O(N^{-1}),\\
  \pr_{nj}^2F_{s,N} &= \pr_j\gamma_{s,x}(Nx_n)+O(N^{-1}), \qquad j\leq n-1,\\
  \pr_{nn}^2F_{s,N} &= N\pr_t\gamma_{s,x}(Nx_n) +2\pr_n\gamma_{s,x}(Nx_n) +\sigma_{s,x}(Nx_n) +O(N^{-1}).
\end{align*}
The first-derivative formulas show that, for every $\delta>0$, the frequency may be chosen so large that
\begin{equation}\label{eq:local-first-derivative}
  \|F_{s,N}\|_{C^0}+\sum_{j=1}^{n-1}\|\pr_jF_{s,N}\|_{C^0}+\|\pr_nF_{s,N}-\gamma_{s,x}(Nx_n)\|_{C^0}<\delta
\end{equation}
uniformly in $s$. Although this frequency may depend on the spatial derivatives of the loops, $|\gamma_{s,x}(t)|\leq C|a(s,x)|$. Taking $\delta=1$ gives
\[
  \sup_{s\in[0,1]}\|F_{s,N}\|_{C^1}\leq C\left(1+\|a\|_{C^0}\right),
\]
where $C$ is uniform for uniformly equivalent families and is independent of the higher derivatives of $h_s$ and $a$.

For the perturbation \eqref{eq:perturbed-metric},
\begin{align}
  \overline\nabla\bigl((h_s)_{F_{s,N}}-h_s\bigr)
  ={}&2\sum_{q=1}^2e^{2F_{s,N}^q}\dd F_{s,N}^q\otimes\theta^q\otimes\theta^q +\sum_{q=1}^2\bigl(e^{2F_{s,N}^q}-1\bigr)\overline\nabla(\theta^q\otimes\theta^q).
  \label{eq:local-metric-first-derivative}
\end{align}
The hypotheses of \eqref{eq:local-uniform-first-derivative} uniformly bound the coframes and their first derivatives. Thus \eqref{eq:local-first-derivative} and \eqref{eq:local-metric-first-derivative} give \eqref{eq:local-uniform-first-derivative}.

We now substitute these expansions into \eqref{eq:local-identity}.  Since $\gamma_1+\gamma_2=0$,
\[
  \pr_t\gamma_1+\pr_t\gamma_2=0, \qquad \pr_j\gamma_1+\pr_j\gamma_2=0 \quad(j=1,\ldots,n).
\]
It follows that the order-$N$ terms, the $2\pr_n\gamma$ terms, and the leading mixed second-derivative terms cancel.  Moreover, $\gamma_1^2+\gamma_2^2+\gamma_1\gamma_2 = q_{s,x}^2$. Since $\Psi_s(x,0)=0$ and $\Psi_s,\Psi_{1,s},\Psi_{2,s}$ are smooth,
\begin{align*}
  \Psi_s\bigl(x,\TT(F_{s,N})\bigr) &= O(N^{-1}),\\
  \Psi_{\ell,s}\bigl(x,\TT(F_{s,N})\bigr) &= \Psi_{\ell,s}(x,0)+O(N^{-1}), \qquad \ell=1,2.
\end{align*}
Using \eqref{eq:local-identity}, we obtain
\begin{align*}
  {}& \Scal_{(h_s)_{F_{s,N}}} -\Scal_{h_s} +a^2\\
  ={}& -2\alpha_s^2 \bigl( (\sigma_{s,x})_1+(\sigma_{s,x})_2 \bigr) -2\alpha_s^2 \bigl( q_{s,x}^2+(-q_{s,x})^2-q_{s,x}^2 \bigr)\\
  &+\Psi_{1,s}(x,0)q_{s,x} -\Psi_{2,s}(x,0)q_{s,x} +a^2 +O(N^{-1})\\
  ={}& -2\alpha_s^2(\sigma_{s,x})_1 -2\alpha_s^2q_{s,x}^2 +\bigl[ \Psi_{1,s}(x,0)-\Psi_{2,s}(x,0) \bigr]q_{s,x} +a^2+O(N^{-1})\\
  ={}& O(N^{-1}),
\end{align*}
where the last equality follows from \eqref{eq:delta-loop}. Equations \eqref{eq:convex-integration} and \eqref{eq:perturbed-metric} therefore give constants independent of $N$ such that
\[
  \sup_{(s,x)\in[0,1]\times U}\left|\Scal_{(h_s)_{F_{s,N}}}-\Scal_{h_s}+a^2\right|\leq\frac{C}{N},\qquad \sup_{s\in[0,1]}\|(h_s)_{F_{s,N}}-h_s\|_{C^0}\leq\frac{C'}{N}.
\]
Choose $N$ sufficiently large and set $\widetilde h_s:=(h_s)_{F_{s,N}}$. This proves \eqref{eq:local-c0} and \eqref{eq:local-scal}. Outside $\supp_x a$, both loops vanish, so $F_{s,N}=0$; if $a(0,\cdot)=0$, then $F_{0,N}=0$. The remaining conclusions follow.
\end{proof}

We now apply the local deformation successively on a locally finite family of coordinate charts.  The following proposition is the parameterized version of \cite{ART}*{Theorem~1.1} needed here.  Its unparameterized form was first proved by Lohkamp \cite{Lohkamp}.

\begin{proposition}\label{prop:global-path}
Let $(M,g)$ be a smooth Riemannian $n$-manifold, and let $k\in C^\infty(M)$ be positive.  Given continuous functions $\varepsilon,\eta:M\to(0,\infty)$, there is a smooth family of metrics $(g_s)_{s\in[0,1]}$, with $g_0=g$, such that
\begin{align*}
  |g_s-g|_g(x) &< \varepsilon(x),\\
  \left| \Scal_{g_s}(x)-\Scal_g(x)+s^2k(x) \right| &< \eta(x)
\end{align*}
for every $s\in[0,1]$ and $x\in M$.
Moreover, suppose that $g_i$ are locally uniformly equivalent and bounded in $C^1_{\mathrm{loc}}$, and that $k_i>0$ are bounded in $C^0_{\mathrm{loc}}$. For arbitrary positive continuous $\varepsilon_i,\eta_i$, the corresponding paths may be chosen so that, for every $K\Subset M$,
\begin{equation}\label{eq:global-first-derivative}
  \sup_i\sup_{s\in[0,1]}\|g_{i,s}\|_{W^{1,\infty}(K;\overline g)}\leq C_K<\infty.
\end{equation}
\end{proposition}

\begin{proof}
Choose a countable, locally finite cover $\{U_j\}$ of $M$ by precompact coordinate domains and nonnegative $\rho_j\in C_c^\infty(U_j)$ with no common zero. Set
\begin{equation}\label{eq:quadratic-partition}
  \psi_j:=\rho_j(\sum_k\rho_k^2)^{-1/2}\in C_c^\infty(U_j),\qquad \sum_{j=1}^\infty\psi_j^2=1.
\end{equation}

Take $g_s^{(0)}=g$ and $a_j(s,x):=s\psi_j(x)\sqrt{k(x)}$. By Lemma~\ref{lem:supported-deformation} and the compactness of $\supp\psi_j$, we may inductively choose $g_s^{(j)}=g_s^{(j-1)}$ outside $\supp\psi_j$ such that
\begin{align}
  |g_s^{(j)}-g_s^{(j-1)}|_g(x) &< 2^{-j-1}\varepsilon(x), \label{eq:metric-error-j}\\
  \left| \Scal_{g_s^{(j)}} -\Scal_{g_s^{(j-1)}} +a_j^2 \right|(x) &< 2^{-j-1}\eta(x). \label{eq:scalar-error-j}
\end{align}

Local finiteness gives a smooth $g_s:=\lim_{j\to\infty}g_s^{(j)}$. Summing \eqref{eq:metric-error-j} gives
\[
  |g_s-g|_g(x) \leq \sum_{j=1}^\infty |g_s^{(j)}-g_s^{(j-1)}|_g(x) < \sum_{j=1}^\infty2^{-j-1}\varepsilon(x) < \varepsilon(x).
\]
Since $a_j(0,x)=0$ at every step, $g_0=g$. Similarly, \eqref{eq:scalar-error-j} and \eqref{eq:quadratic-partition} give
\[
  \Scal_{g_s}-\Scal_g = -\sum_{j=1}^\infty a_j^2+E_s = -s^2k\sum_{j=1}^\infty\psi_j^2+E_s = -s^2k+E_s,
\]
where $|E_s(x)|<\sum_{j=1}^\infty2^{-j-1}\eta(x)<\eta(x)$. 

For the uniform assertion, use the same $U_j,\psi_j$ and choose a continuous function $\lambda>0$ such that $g_i\geq2\lambda\,\overline g$. At the $j$-th step also impose
\[
  |g_{i,s}^{(j)}-g_{i,s}^{(j-1)}|_{\overline g}(x)<2^{-j-1}\lambda(x),
\]
so the intermediate metrics remain locally uniformly equivalent. For $K\Subset M$, let $J:=\max\{j:K\cap\supp\psi_j\neq\varnothing\}$. On $K\cup\bigcup_{j=1}^J\overline U_j$, the hypotheses and \eqref{eq:local-uniform-first-derivative} give inductively a uniform $C^1$-bound through step $J$. All later deformations vanish on $K$, proving \eqref{eq:global-first-derivative}.
\end{proof}

We remark that the endpoint $s=1$ of Proposition~\ref{prop:global-path} is the almost-prescription theorem \cite{ART}*{Theorem~1.1}, first proved by Lohkamp \cite{Lohkamp}.

\section{Proof of Theorem~\ref{thm:main} when \texorpdfstring{$\kappa=0$}{kappa = 0}}\label{sec:zero-case}
\subsection{The closed case with \texorpdfstring{$\kappa=0$}{}}\label{sec:closed-flat}

\begin{lemma}\label{lem:positive-solutions}
Let $h_i$ be smooth metrics on a closed manifold $M$ converging uniformly to $g$. If $u_i>0$ satisfies
\begin{equation}\label{eq:zero-mode-equation}
  -c_n\Delta_{h_i}u_i+V_iu_i=0
\end{equation}
with $\|V_i\|_{C^0}\to0$, then $u_i\to1$ uniformly after rescaling.
\end{lemma}

\begin{proof}
Normalize $u_i$ so that $\|u_i\|_{L^2(h_i)}=1$. Integrating \eqref{eq:zero-mode-equation} against $u_i$ gives
\[
  c_n\int_M|\nabla u_i|_{h_i}^2\dd\mu_{h_i}=-\int_MV_iu_i^2\dd\mu_{h_i}\leq\|V_i\|_{C^0}\longrightarrow0.
\]
Because $h_i\to g$ uniformly, the metrics are uniformly equivalent and have uniformly bounded Poincar\'e constants. Let $\overline u_i^{h_i}$ be the $h_i$-average of $u_i$ over $M$. Then $\|u_i-\overline u_i^{h_i}\|_{L^2(h_i)}\to0$. Since
\[
  1=\|u_i-\overline u_i^{h_i}\|_{L^2(h_i)}^2+\operatorname{Vol}_{h_i}(M)(\overline u_i^{h_i})^2
\]
and $\overline u_i^{h_i}>0$, we have $\overline u_i^{h_i}\to\operatorname{Vol}_g(M)^{-1/2}>0$.

In a fixed finite atlas, \eqref{eq:zero-mode-equation} is a uniformly elliptic equation in divergence form with uniformly bounded zeroth-order coefficients. The local boundedness and interior H\"older estimates \cite{GT}*{Theorems~8.17 and~8.24} give constants $\gamma\in(0,1)$ and $C>0$, independent of $i$, such that
\[
  \|u_i\|_{L^\infty(M)}+[u_i]_{C^{0,\gamma}(M)}\leq C.
\]
If $\|u_i-\overline u_i^{h_i}\|_{C^0}\not\to0$, then, after passing to a subsequence, $|u_i(x_i)-\overline u_i^{h_i}|\geq\varepsilon>0$ for some $x_i\in M$. The uniform H\"older bound implies $|u_i-\overline u_i^{h_i}|\geq\varepsilon/2$ on balls of uniform positive $h_i$-volume, contradicting $\|u_i-\overline u_i^{h_i}\|_{L^2(h_i)}\to0$. Thus $u_i/\overline u_i^{h_i}\to1$ uniformly.
\end{proof}

\begin{proof}[Proof of Theorem~\ref{thm:main} for closed $M$ and $\kappa=0$]
If $\Scal_{g_0} \equiv 0$, set $\widehat{g}_i = g_0$. Assume $\Scal_{g_0} \not\equiv 0$. The first eigenvalue of $L_{g}$ is given by
\begin{equation}\label{eq:rayleigh}
  \lambda_1(g) = \inf_{u \neq 0} \frac{\int_M \left(c_n |\nabla u|_g^2 + \Scal_g u^2\right) \dd\mu_g}{\int_M u^2 \dd\mu_g}.
\end{equation}
Since $\Scal_{g_0} \ge 0$ and $\Scal_{g_0} \not\equiv 0$, we know $\lambda_1(g_0) > 0$.

For each $i \ge 1$, apply Proposition~\ref{prop:global-path} to $g_0$ with $k_i = \Scal_{g_0} + \frac{1}{i}$, $\varepsilon_i = \frac{1}{i}$, and $\eta_i = \frac{1}{4i}$ to obtain metrics $(g_{i,s})_{s \in [0,1]}$ satisfying
\begin{align}
  g_{i,0} &= g_0, \label{eq:gi0} \\
  \sup_{s \in [0,1]} \|g_{i,s} - g_0\|_{C^0_{g_0}} &< \frac{1}{i}, \label{eq:gi-close} \\
  \left\| \Scal_{g_{i,s}} - \left[ (1-s^2)\Scal_{g_0} - \frac{s^2}{i} \right] \right\|_{C^0} &< \frac{1}{4i}. \label{eq:gi-Scal}
\end{align}
Since $g_0$ is fixed and $k_i$ is uniformly bounded, \eqref{eq:global-first-derivative} gives
\begin{equation}\label{eq:closed-zero-first-derivative}
  \sup_i\sup_{s\in[0,1]}\|g_{i,s}\|_{W^{1,\infty}(M;\overline g)}<\infty.
\end{equation}

At $s = 1$, bound \eqref{eq:gi-Scal} gives $\Scal_{g_{i,1}} < -\frac{3}{4i} < 0$. Testing \eqref{eq:rayleigh} with a constant function yields $\lambda_1(g_{i,1}) < 0$. Since $\lambda_1(g_{i,0}) > 0$, continuity in $s$ provides a root $s_i \in (0,1)$ with
\[
  \lambda_1(g_{i,s_i}) = 0.
\]

To show that $s_i \to 1$, fix $\delta \in (0,1)$ and set $c_\delta := 1 - (1-\delta)^2 > 0$. For $s \le 1 - \delta$, estimate \eqref{eq:gi-Scal} gives
\begin{equation}\label{eq:potential-lower}
  \Scal_{g_{i,s}} \ge c_\delta \Scal_{g_0} - \frac{5}{4i}.
\end{equation}
Let $\Lambda_\delta > 0$ be the first eigenvalue of $-c_n \Delta_{g_0} + c_\delta \Scal_{g_0}$. By \eqref{eq:gi-close}, we have, for every covector $\xi$ and every $i \ge 4n$, uniformly in $s$ and $x$,
\begin{align*}
    \left(1-\frac{2n}{i}\right)|\xi|_{g_0}^2 \le |\xi|_{g_{i,s}}^2 \le \left(1+\frac{2n}{i}\right)|\xi|_{g_0}^2,\\
    \left(1-\frac{2n}{i}\right)\dd\mu_{g_0} \le \dd\mu_{g_{i,s}} \le \left(1+\frac{2n}{i}\right)\dd\mu_{g_0}.
\end{align*}
Thus, for $s \le 1-\delta$, we bound the numerator in Rayleigh quotient using \eqref{eq:potential-lower}:
\begin{align*}
  {}&\int_M \left( c_n |\nabla u|_{g_{i,s}}^2 + \Scal_{g_{i,s}} u^2 \right) \dd\mu_{g_{i,s}}\\
  \ge{}& \left(1-\frac{2n}{i}\right)^2 \int_M \left( c_n |\nabla u|_{g_0}^2 + c_\delta \Scal_{g_0} u^2 \right) \dd\mu_{g_0} - \frac{5}{4i} \int_M u^2 \dd\mu_{g_{i,s}} \\
  \ge{}& \left[ \frac{\left(1 - \frac{2n}{i}\right)^2}{1 + \frac{2n}{i}} \Lambda_\delta - \frac{5}{4i} \right] \int_M u^2 \dd\mu_{g_{i,s}}.
\end{align*}
For large $i$, the bracketed coefficient is strictly positive, forcing $\lambda_1(g_{i,s}) > 0$ and $s_i > 1 - \delta$. Thus $s_i \longrightarrow 1$.

Set $h_i := g_{i,s_i}$. Then $h_i \to g_0$ uniformly in $C^0$, and 
\[
\|\Scal_{h_i}\|_{C^0} \le (1 - s_i^2)\|\Scal_{g_0}\|_{C^0} + \frac{5}{4i} \longrightarrow 0.
\]

Let $u_i>0$ be the first eigenfunction of $L_{h_i}$, satisfying $L_{h_i}u_i=0$. By Lemma~\ref{lem:positive-solutions}, normalize $u_i$ so that $u_i\to1$ uniformly. Equation \eqref{eq:closed-zero-first-derivative}, $\Scal_{h_i}\to0$, and Lemma~\ref{lem:conformal-regularity} give $u_i\to1$ in $C^1$. Setting $\widehat g_i:=u_i^{4/(n-2)}h_i$, we obtain $\Scal_{\widehat g_i}=0$, $\widehat g_i\to g_0$ uniformly, and a uniform $W^{1,\infty}$-bound. Lemma~\ref{lem:holder-compactness} completes the proof.
\end{proof}
\subsection{The open case with \texorpdfstring{$\kappa=0$}{}}\label{sec:open-flat}

For a Riemannian metric $h$, define the quadratic form
\begin{equation}\label{eq:quadratic-form}
  \mathcal{Q}_h[u] := \int_M \left(c_n |\nabla u|_h^2 + \Scal_h u^2\right) \dd\mu_h, \qquad u \in C_c^\infty(M).
\end{equation}

\begin{definition}\label{def:parabolic}
A metric $h$ on a noncompact manifold $M$ is \emph{parabolic} if, for every compact set $K \Subset M$, there exist cutoff functions $\chi_j \in C_c^\infty(M)$ satisfying
\[
  0 \le \chi_j \le 1, \qquad \chi_j \equiv 1 \text{ near } K, \qquad \int_M |\nabla \chi_j|_h^2 \dd\mu_h \longrightarrow 0 \quad \text{as } j \to \infty.
\]
No completeness is assumed in this definition.
\end{definition}

\begin{lemma}\label{lem:parabolicization}
Let $M$ be a connected, noncompact manifold, and let $g$ be a smooth metric with $\Scal_g\geq0$ and $\Scal_g\not\equiv0$. For every compact set $K\Subset M$, there exists a positive smooth function $\varphi$ such that $q:=\varphi^{4/(n-2)}g$ is parabolic, $q=g$ near $K$, and $\Scal_q\geq0$ with $\Scal_q\not\equiv0$.
\end{lemma}

\begin{proof}
Choose a connected precompact smooth domain $\Omega\supset K$ containing a point where $\Scal_g>0$. By adjoining its finitely many relatively compact complementary components and smoothing the boundary if necessary, we may assume that $M\setminus\overline\Omega$ has no relatively compact components. Choose a smooth connected exhaustion $\overline\Omega\Subset\Omega_1\Subset\Omega_2\Subset\cdots\Subset M$ such that every component of $A_j:=\Omega_j\setminus\overline\Omega$ meets both $\partial\Omega$ and $\partial\Omega_j$. Classical Dirichlet theory \cite{GT}*{Theorem~6.14} gives a smooth solution of
\begin{equation}\label{eq:exterior-dirichlet}
  L_gv_j=0\quad\text{in }A_j,\qquad v_j=1\quad\text{on }\partial\Omega,\qquad v_j=0\quad\text{on }\partial\Omega_j.
\end{equation}
The weak and strong maximum principles \cite{GT}*{Theorems~3.1 and~3.5} give $0<v_j<1$, and comparison gives $v_j\leq v_{j+1}$. The interior Schauder estimates \cite{GT}*{Theorem~6.2 and Corollary~6.3}, the boundary estimate on a fixed collar of $\partial\Omega$ \cite{GT}*{Corollary~6.7}, and interior bootstrapping \cite{GT}*{Theorem~6.17} give a smooth limit
\[
  v_j\longrightarrow v\quad\text{locally in }C^\infty(M\setminus\overline\Omega)\quad\text{and in }C^1\text{ up to }\partial\Omega,
\]
where
\begin{equation}\label{eq:minimal-potential}
  L_gv=0,\qquad 0<v\leq1,\qquad v=1\quad\text{on }\partial\Omega.
\end{equation}
Choose a smooth function $F:[0,1]\to[0,1]$ satisfying, for some $A\geq1$,
\begin{equation}\label{eq:concave-F}
\begin{cases}
  F(0)=0,\qquad F\equiv1\text{ near }1,\qquad 0<F(t)\leq At\quad\text{on }(0,1],\\
  F'\geq0,\qquad F''\leq0,\qquad F(t)-tF'(t)\geq0\quad\text{on }[0,1].
\end{cases}
\end{equation}
For example, let $\vartheta:[0,1]\to[0,1]$ be smooth and nonincreasing with $\vartheta\equiv1$ near $0$ and $\vartheta\equiv0$ near $1$, and set
\[
  F(t):=A\int_0^t\vartheta(r)\dd r,\qquad A:=\left(\int_0^1\vartheta(r)\dd r\right)^{-1}.
\]
Define $\varphi:=1$ on $\Omega$ and $\varphi:=F(v)$ on $M\setminus\Omega$. Since $F\equiv1$ near $1$, the function $\varphi$ is smooth across $\partial\Omega$, and $0<\varphi\leq Av$ on $M$, where $v$ is extended by $1$ on $\Omega$. In particular, $q:=\varphi^{4/(n-2)}g$ agrees with $g$ near $K$. On $\Omega$,
\[
  \Scal_q=\varphi^{-\frac{n+2}{n-2}}L_g\varphi=L_g(1)=\Scal_g\geq0,
\]
and $\Scal_q\not\equiv0$ because $\Omega$ contains a point where $\Scal_g>0$. On $M\setminus\overline\Omega$, \eqref{eq:minimal-potential} and \eqref{eq:concave-F} give
\begin{align*}
  \varphi^{\frac{n+2}{n-2}}\Scal_q
  &=L_gF(v)\\
  &=-c_n\Delta_gF(v)+\Scal_gF(v)\\
  &=-c_nF''(v)|\nabla v|_g^2-c_nF'(v)\Delta_gv+\Scal_gF(v)\\
  &=-c_nF''(v)|\nabla v|_g^2+\Scal_g\bigl(F(v)-vF'(v)\bigr)\geq0.
\end{align*}

Finally, set
\[
  w_j:=
  \begin{cases}
    1 & \text{on }\Omega,\\
    v_j/v & \text{on }A_j,\\
    0 & \text{on }M\setminus\Omega_j.
  \end{cases}
\]
Fix a smooth nondecreasing function $\eta:[0,1]\to[0,1]$ with $\eta\equiv0$ near $0$ and $\eta\equiv1$ on $[1/2,1]$, and set $\chi_j:=\eta(w_j)$. Since $\eta$ is constant near $0$ and $1$, we have $\chi_j\in C_c^\infty(M)$ and $0\leq\chi_j\leq1$. Moreover, $\chi_j\equiv1$ near every fixed compact set for all sufficiently large $j$.

Let $\nu_j$ be the outward unit normal of $\partial A_j$ and set $C_\eta:=A^2\|\eta'\|_\infty^2$. Using $\varphi\leq Av$, $w_j=0$ on $\partial\Omega_j$, $w_j=v_j=v=1$ on $\partial\Omega$, and $L_gv_j=L_gv=0$, we obtain
\begin{align*}
  \int_M|\nabla\chi_j|_q^2\dd\mu_q
  &=\int_M\varphi^2|\nabla\chi_j|_g^2\dd\mu_g\\
  &\leq A^2\int_Mv^2|\nabla\chi_j|_g^2\dd\mu_g\\
  &\leq C_\eta\int_{A_j}v^2|\nabla w_j|_g^2\dd\mu_g\\
  &=-C_\eta\int_{A_j}w_j\operatorname{div}_g\left(v^2\nabla\left(\frac{v_j}{v}\right)\right)\dd\mu_g+C_\eta\int_{\partial A_j}w_jv^2\partial_{\nu_j}\left(\frac{v_j}{v}\right)\dd\sigma_g\\
  &=-C_\eta\int_{A_j}w_j\left(v\Delta_gv_j-v_j\Delta_gv\right)\dd\mu_g+C_\eta\int_{\partial\Omega}\left(\partial_{\nu_j}v_j-\partial_{\nu_j}v\right)\dd\sigma_g\\
  &=C_\eta\int_{\partial\Omega}\left(\partial_{\nu_j}v_j-\partial_{\nu_j}v\right)\dd\sigma_g\longrightarrow0.
\end{align*}
The last limit follows from the $C^1$ convergence of $v_j$ to $v$ up to $\partial\Omega$. Thus $q$ is parabolic by Definition~\ref{def:parabolic}.
\end{proof}

The next lemma supplies the weighted spectral gap needed along the scalar curvature decreasing path and the positive solutions used for the conformal correction. The first assertion is the subcritical alternative for nonnegative Schr\"odinger operators; see \cite{PinchoverTintarev}*{Theorem~1.4}. We include a direct proof.

\begin{lemma}\label{lem:parabolic-tools}
Let $q$ be a parabolic metric on a connected, noncompact manifold $M$ with $\Scal_q\geq0$ and $\Scal_q\not\equiv0$. There exists a positive smooth function $\beta$ on $M$ such that
\begin{equation}\label{eq:weighted-gap}
  \int_M\beta u^2\dd\mu_q\leq\mathcal Q_q[u]\qquad\text{for all }u\in C_c^\infty(M).
\end{equation}
Moreover, suppose $h_i\to q$ in $C^0_{\mathrm{loc}}$, $\Scal_{h_i}\to0$ in $C^0_{\mathrm{loc}}$, and $\mathcal Q_{h_i}\geq0$ on $C_c^\infty(M)$ for every $i$. For every $p\in M$, there exist positive smooth functions $u_i$ such that
\[
  L_{h_i}u_i=0,\qquad u_i(p)=1,\qquad u_i\longrightarrow1\quad\text{in }C^0_{\mathrm{loc}}.
\]
\end{lemma}

\begin{proof}
Choose a connected precompact smooth domain $B\Subset M$ on which $\Scal_q>0$. Choose a smooth locally finite partition of unity $\{\rho_j\}$ with compact supports and a smooth connected exhaustion $B\Subset D_1\Subset D_2\Subset\cdots\Subset M$ such that $\operatorname{supp}\rho_j\Subset D_j$.

For every $j$, there exists $\kappa_j>0$ such that
\begin{equation}\label{eq:local-coercivity}
  \mathcal Q_q[u]\geq\kappa_j\int_{D_j}u^2\dd\mu_q\qquad\text{for all }u\in C_c^\infty(M).
\end{equation}
Otherwise, there are $u_\ell\in C_c^\infty(M)$ such that $\int_{D_j}u_\ell^2\dd\mu_q=1$ and $\mathcal Q_q[u_\ell]\to0$. Let $\overline{u}^q_{\ell,j}$ be the $q$-average of $u_\ell$ on $D_j$. Since $\Scal_q\geq0$, the Poincar\'e inequality gives
\[
  \int_{D_j}|u_\ell-\overline{u}^q_{\ell,j}|^2\dd\mu_q\leq\frac{C_j}{c_n}\mathcal Q_q[u_\ell]\longrightarrow0.
\]
Moreover,
\[
  1=\int_{D_j}|u_\ell-\overline{u}^q_{\ell,j}|^2\dd\mu_q+(\overline{u}^q_{\ell,j})^2\mu_q(D_j),
\]
so, after passing to a subsequence, $\overline{u}^q_{\ell,j}\to a\neq0$, as $\ell\rightarrow \infty$. Since $B\Subset D_j$,
\[
  \mathcal Q_q[u_\ell]\geq\int_B\Scal_qu_\ell^2\dd\mu_q\longrightarrow a^2\int_B\Scal_q\dd\mu_q>0,
\]
a contradiction. This proves \eqref{eq:local-coercivity}.

Set $\beta:=\sum_{j=1}^\infty2^{-j}\kappa_j\rho_j$. Then $\beta>0$ is smooth, and
\[
  \int_M\beta u^2\dd\mu_q\leq\sum_{j=1}^\infty2^{-j}\kappa_j\int_{D_j}u^2\dd\mu_q\leq\sum_{j=1}^\infty2^{-j}\mathcal Q_q[u]=\mathcal Q_q[u].
\]

For the second assertion, fix $p\in M$ and a smooth connected exhaustion $p\in\Omega_1\Subset\Omega_2\Subset\cdots\Subset M$. Since $\mathcal Q_{h_i}\geq0$, the first Dirichlet eigenvalue satisfies
\[
  \lambda_1(L_{h_i}+j^{-1};\Omega_j)=\inf_{\zeta\in H_0^1(\Omega_j)\setminus\{0\}}\frac{\mathcal Q_{h_i}[\zeta]+j^{-1}\int_{\Omega_j}\zeta^2\dd\mu_{h_i}}{\int_{\Omega_j}\zeta^2\dd\mu_{h_i}}\geq\frac1j.
\]
Variational Dirichlet theory \cite{GT}*{Section~8.2} therefore gives a unique smooth solution
\[
  (L_{h_i}+j^{-1})z_{i,j}=0\quad\text{in }\Omega_j,\qquad z_{i,j}=1\quad\text{on }\partial\Omega_j.
\]
Integration against the negative part of $z_{i,j}$ gives $z_{i,j}\geq0$, and the Harnack inequality \cite{GT}*{Theorem~8.20} gives $z_{i,j}>0$. Set $u_{i,j}:=z_{i,j}/z_{i,j}(p)$. For fixed $i$, the Harnack inequality and interior elliptic estimates \cite{GT}*{Theorems~6.2 and~6.17} give, after passing to a diagonal subsequence, a positive smooth limit $u_i$ satisfying $ L_{h_i}u_i=0$ and $u_i(p)=1$.

We now let $i\to\infty$. On every connected compact set containing $p$, a finite chain of overlapping coordinate balls and the uniform Harnack inequality propagate $u_i(p)=1$ to uniform positive upper and lower bounds. The Harnack constants are uniform because $h_i\to q$ locally uniformly and $\Scal_{h_i}\to0$ in $C^0_{\mathrm{loc}}$, which give common local ellipticity and zeroth-order bounds. The local H\"older estimate \cite{GT}*{Theorems~8.20 and~8.24} then gives uniform local $C^{0,\alpha}$-bounds for $u_i$. Integrating $L_{h_i}u_i=0$ against $\eta^2u_i$, where $\eta\in C_c^\infty(M)$, gives
\[
  \int_M\eta^2|\nabla u_i|_{h_i}^2\dd\mu_{h_i}\leq C\int_M\left(u_i^2|\nabla\eta|_{h_i}^2+|\Scal_{h_i}|\eta^2u_i^2\right)\dd\mu_{h_i}.
\]
Thus $u_i$ is uniformly bounded in $H^1_{\mathrm{loc}}$. Every subsequence has a further subsequence converging locally uniformly and weakly in $H^1_{\mathrm{loc}}$ to a positive function $u$ with $u(p)=1$. For every $\psi\in C_c^\infty(M)$,
\[
  0=\int_M\left(c_n\langle\nabla u_i,\nabla\psi\rangle_{h_i}+\Scal_{h_i}u_i\psi\right)\dd\mu_{h_i}\longrightarrow c_n\int_M\langle\nabla u,\nabla\psi\rangle_q\dd\mu_q.
\]
Hence $\Delta_qu=0$, and $u$ is smooth by elliptic regularity. Integrating $\Delta_qu=0$ against $\chi^2/u$, where $\chi\in C_c^\infty(M)$, gives
\begin{align*}
    \int_M\chi^2|\nabla\log u|_q^2\dd\mu_q&=2\int_M\chi\langle\nabla\log u,\nabla\chi\rangle_q\dd\mu_q\\
    &\leq\frac12\int_M\chi^2|\nabla\log u|_q^2\dd\mu_q+2\int_M|\nabla\chi|_q^2\dd\mu_q.
\end{align*}
Therefore,
\[
  \int_M\chi^2|\nabla\log u|_q^2\dd\mu_q\leq4\int_M|\nabla\chi|_q^2\dd\mu_q.
\]
For every $K\Subset M$, applying this inequality to parabolic cutoff functions $\chi_k\equiv1$ near $K$ gives $\int_K|\nabla\log u|_q^2\dd\mu_q=0$. Thus $u\equiv u(p)=1$. Since every subsequential limit is $1$, we conclude that $u_i\to1$ in $C^0_{\mathrm{loc}}$.
\end{proof}

\begin{proposition}\label{prop:parabolic-case}
Let $q$ be a parabolic metric on a connected, noncompact manifold $M$ satisfying $\Scal_q\geq0$ and $\Scal_q\not\equiv0$. Then $q$ is a limit in the local Riemannian Burnett compactness class of smooth scalar-flat metrics.
\end{proposition}

\begin{proof}
Let $\beta>0$ be given by Lemma~\ref{lem:parabolic-tools}, so that
\begin{equation}\label{eq:weighted-gap-recalled}
  \int_M\beta u^2\dd\mu_q\leq\mathcal Q_q[u]
  \qquad\text{for all }u\in C_c^\infty(M).
\end{equation}
For each $i\geq1$, apply Proposition~\ref{prop:global-path} with
\[
  k_i:=\Scal_q+\frac1i\beta,\qquad
  \varepsilon_i:=\frac1i,\qquad
  \eta_i:=\frac{1}{4i^2}\beta.
\]
This gives metrics $(q_{i,s})_{s\in[0,1]}$ satisfying $q_{i,0}=q$ and
\begin{align}
  |q_{i,s}-q|_q&<\frac1i,\label{eq:parabolic-path-metric}\\
  \left|\Scal_{q_{i,s}}-\left((1-s^2)\Scal_q-\frac{s^2}{i}\beta\right)\right|
  &<\frac{1}{4i^2}\beta.\label{eq:parabolic-path-scalar}
\end{align}
Since $q$ is fixed and $k_i$ is locally uniformly bounded, \eqref{eq:global-first-derivative} gives for every $K\Subset M$
\begin{equation}\label{eq:parabolic-path-first-derivative}
  \sup_i\sup_{s\in[0,1]}
  \|q_{i,s}\|_{W^{1,\infty}(K;\overline g)}<\infty.
\end{equation}

For $i\geq4n$, \eqref{eq:parabolic-path-metric} places the relative eigenvalues of $q_{i,s}$ with respect to $q$ in $[1-i^{-1},1+i^{-1}]$. Comparing the inverse metrics and volume forms gives
\begin{align*}
  \left(1-\frac{2n}{i}\right)\dd\mu_q
  &\leq\dd\mu_{q_{i,s}}
  \leq\left(1+\frac{2n}{i}\right)\dd\mu_q,\\
  \left(1-\frac{2n}{i}\right)|\nabla u|_q^2\dd\mu_q
  &\leq|\nabla u|_{q_{i,s}}^2\dd\mu_{q_{i,s}}
  \leq\left(1+\frac{2n}{i}\right)|\nabla u|_q^2\dd\mu_q.
\end{align*}

We first use $\beta>0$ to make the quadratic form negative at $s=1$. By \eqref{eq:parabolic-path-scalar},
\[
  \Scal_{q_{i,1}}
  \leq-\left(\frac1i-\frac{1}{4i^2}\right)\beta.
\]
Fix a nonempty precompact open set $B\Subset M$ and set $b:=\int_B\beta\dd\mu_q>0$. By parabolicity, choose $\chi_i\in C_c^\infty(M)$ equal to $1$ near $\overline B$ such that
\[
  c_n\int_M|\nabla\chi_i|_q^2\dd\mu_q<\frac{b}{8i}.
\]
Since $\int_M\beta\chi_i^2\dd\mu_q\geq b$, for all sufficiently large $i$,
\begin{align}
  \mathcal Q_{q_{i,1}}[\chi_i]
  &\leq\left(1+\frac{2n}{i}\right)c_n
  \int_M|\nabla\chi_i|_q^2\dd\mu_q -\left(1-\frac{2n}{i}\right)
  \left(\frac1i-\frac{1}{4i^2}\right)
  \int_M\beta\chi_i^2\dd\mu_q\notag\\
  &<\left(1+\frac{2n}{i}\right)\frac{b}{8i}
  -\left(1-\frac{2n}{i}\right)
  \left(\frac1i-\frac{1}{4i^2}\right)b<0.
  \label{eq:negative-endpoint}
\end{align}
Thus the negative potential $-i^{-1}\beta$ dominates the arbitrarily small Dirichlet energy of the parabolic cutoff.

At $s=0$, we have $\mathcal Q_{q_{i,0}}=\mathcal Q_q\geq0$ because $\Scal_q\geq0$. Define
\begin{equation}\label{eq:first-loss}
  s_i:=\inf\left\{s\in[0,1]:
  \mathcal Q_{q_{i,s}}[u]<0
  \text{ for some }u\in C_c^\infty(M)\right\}.
\end{equation}
By \eqref{eq:negative-endpoint} and continuity in $s$, we have $s_i<1$ and
\begin{equation}\label{eq:threshold-nonnegative}
  \mathcal Q_{q_{i,s_i}}[u]\geq0
  \qquad\text{for all }u\in C_c^\infty(M).
\end{equation}

We next use the weighted gap to prove that $s_i\to1$. Fix $\delta\in(0,1)$. For $s\leq1-\delta$, the metric comparisons and \eqref{eq:parabolic-path-scalar} give
\begin{align}
  \mathcal Q_{q_{i,s}}[u]
  \geq{}&
  \left(1-\frac{2n}{i}\right)
  \int_M\left(c_n|\nabla u|_q^2+(1-s^2)\Scal_q u^2\right)\dd\mu_q\notag\\
  &-\left(1+\frac{2n}{i}\right)
  \left(\frac1i+\frac{1}{4i^2}\right)
  \int_M\beta u^2\dd\mu_q\notag\\
  \geq{}&
  \left[
  \left(1-\frac{2n}{i}\right)(1-s^2)
  -\left(1+\frac{2n}{i}\right)
  \left(\frac1i+\frac{1}{4i^2}\right)
  \right]\mathcal Q_q[u].
  \label{eq:threshold-lower}
\end{align}
Indeed, \eqref{eq:weighted-gap-recalled} absorbs the $\beta$-term into $\mathcal Q_q$, while $\Scal_q\geq0$ gives
\[
  c_n|\nabla u|_q^2+(1-s^2)\Scal_q u^2
  \geq(1-s^2)
  \left(c_n|\nabla u|_q^2+\Scal_q u^2\right).
\]
Since $1-s^2\geq1-(1-\delta)^2>0$, the coefficient in \eqref{eq:threshold-lower} is positive for all sufficiently large $i$. Hence $\mathcal Q_{q_{i,s}}\geq0$ for every $s\leq1-\delta$, so $s_i\geq1-\delta$. Therefore $s_i\to1$.

Set $h_i:=q_{i,s_i}$. Equations \eqref{eq:parabolic-path-metric} and \eqref{eq:parabolic-path-scalar} give $h_i\to q$, $\Scal_{h_i}\to$ in $C^0_{\mathrm{loc}}$.
By \eqref{eq:threshold-nonnegative}, $\mathcal Q_{h_i}\geq0$. Lemma~\ref{lem:parabolic-tools} gives positive smooth functions $u_i$ satisfying
\[
  L_{h_i}u_i=0,\qquad
  u_i\longrightarrow1
  \quad\text{in }C^0_{\mathrm{loc}}.
\]
Equation \eqref{eq:parabolic-path-first-derivative} and Lemma~\ref{lem:conformal-regularity} give $u_i\to1$ in $C^1_{\mathrm{loc}}$. Therefore $\widehat q_i:=u_i^{4/(n-2)}h_i$ is scalar-flat, converges to $q$ in $C^0_{\mathrm{loc}}$, and is locally bounded in $W^{1,\infty}$. Lemma~\ref{lem:holder-compactness} completes the proof.
\end{proof}

We now prove Theorem~\ref{thm:main} for open $M$ when $\kappa=0$.

\begin{proof}[Proof of Theorem~\ref{thm:main} for open $M$ and $\kappa=0$]
If $\Scal_{g_0}\equiv0$, take $\widehat g_i:=g_0$ for every $i$. We assume $\Scal_{g_0}\not\equiv0$. Choose a compact exhaustion $K_1\Subset K_2\Subset\cdots\Subset M$. For each $i$, Lemma~\ref{lem:parabolicization} gives a parabolic metric $q_i$ satisfying
\[
  \Scal_{q_i}\ge0,\qquad \Scal_{q_i}\not\equiv0,\qquad q_i=g_0\quad\text{near }K_{i+1}.
\]
Let $\beta_i$ be the positive smooth function associated with $q_i$ given by Lemma~\ref{lem:parabolic-tools}. Since $q_i=g_0$ near $K_{i+1}$, the metrics $q_i$ are locally uniformly equivalent and bounded in $C^1_{\mathrm{loc}}$. For the pairs $(i,m)$ satisfying $\beta_i/m\leq1$ on $K_{i+1}$, the decrements $\Scal_{q_i}+m^{-1}\beta_i$ are bounded in $C^0_{\mathrm{loc}}$. We therefore use the uniform assertion in Proposition~\ref{prop:global-path} when applying the construction in Proposition~\ref{prop:parabolic-case} to these pairs. It gives scalar-flat metrics $\widehat q_{i,m}=u_{i,m}^{4/(n-2)}h_{i,m}$ such that, for each fixed $i$,
\[
  \widehat q_{i,m}\longrightarrow q_i\quad\text{in }C^0_{\mathrm{loc}},\qquad u_{i,m}\longrightarrow1\quad\text{in }C^1_{\mathrm{loc}}.
\]
Choose $m(i)$ sufficiently large that
\[
  \frac{\beta_i}{m(i)}\leq1\quad\text{on }K_{i+1},\qquad \|\widehat q_{i,m(i)}-q_i\|_{C^0(K_i)}<\frac1i,\qquad \|u_{i,m(i)}-1\|_{C^1(K_i)}<\frac1i,
\]
and set $\widehat g_i:=\widehat q_{i,m(i)}$. Then $\widehat g_i\to g_0$ in $C^0_{\mathrm{loc}}$. For every $K\Subset M$, we have $K\subset K_i$ for all sufficiently large $i$, and hence
\[
  \sup_i\|h_{i,m(i)}\|_{W^{1,\infty}(K;\overline g)}<\infty,\qquad
  \sup_i\|u_{i,m(i)}\|_{C^1(K)}<\infty.
\]
Hence, $\widehat g_i=u_{i,m(i)}^{4/(n-2)}h_{i,m(i)}$ satisfies \eqref{eq:main-uniform-first-derivative}, and Lemma~\ref{lem:holder-compactness} completes the proof.
\end{proof}

We remark that this construction need not preserve completeness, since the parabolicizing conformal factor may tend to zero at infinity.
\section{Proof of Theorem~\ref{thm:main} when \texorpdfstring{$\kappa<0$}{kappa < 0}}\label{sec:negative-case}

When $\kappa<0$, only the endpoint of Proposition~\ref{prop:global-path} is needed. Constant sub- and supersolutions then give the conformal correction.

\begin{proof}[Proof of Theorem~\ref{thm:main} when $\kappa<0$]
Fix $g_0\in\MM_{\ge\kappa}(M)$. Apply Proposition~\ref{prop:global-path} with
\[
  k_i:=\Scal_{g_0}-\kappa+\frac{1}{i},\qquad \varepsilon_i:=\frac{1}{i},\qquad \eta_i:=\frac{1}{4i},
\]
and take $s=1$. Since $k_i>0$, this gives a smooth metric $h_i$ such that
\[
  \sup_M|h_i-g_0|_{g_0}<\frac{1}{i},\qquad \left|\Scal_{h_i}-\left(\kappa-\frac{1}{i}\right)\right|<\frac{1}{4i}.
\]
Since $g_0$ is fixed and $k_i$ is locally uniformly bounded, \eqref{eq:global-first-derivative} allows the metrics $h_i$ to be chosen locally uniformly bounded in $W^{1,\infty}$.
In particular,
\begin{equation}\label{eq:scalar-error-explicit}
  -\frac{5}{4i}<\Scal_{h_i}-\kappa<-\frac{3}{4i}<0.
\end{equation}

We seek a positive function $u_i$ such that $\widehat g_i:=u_i^{4/(n-2)}h_i\in \MM_{\kappa}(M)$. By the conformal formula, it is enough to solve
\[
  F_i(u_i):=-c_n\Delta_{h_i}u_i+\Scal_{h_i}u_i-\kappa u_i^{\frac{n+2}{n-2}}=0.
\]
For $i>-5/(2\kappa)$, define
\begin{equation}\label{eq:sub-super-def}
  \underline u_i:=\left(1+\frac{5}{2\kappa i}\right)^{\frac{n-2}{4}},\qquad
  \overline u_i:=\left(1-\frac{5}{2\kappa i}\right)^{\frac{n-2}{4}}.
\end{equation}
Since $\kappa<0$, we have $0<\underline u_i<1<\overline u_i$.  Because $\underline u_i$ and $\overline u_i$ are constant, \eqref{eq:scalar-error-explicit} gives
\begin{align*}
  F_i(\underline u_i)
  &=\underline u_i\left(\Scal_{h_i}-\kappa-\frac{5}{2i}\right)
  <-\frac{13}{4i}\underline u_i<0,\\
  F_i(\overline u_i)
  &=\overline u_i\left(\Scal_{h_i}-\kappa+\frac{5}{2i}\right)
  >\frac{5}{4i}\overline u_i>0.
\end{align*}
Thus $\underline u_i$ and $\overline u_i$ are strict sub- and supersolutions.

If $M$ is closed, choose $A_i>0$ so that $t\mapsto(A_i-\Scal_{h_i})t+\kappa t^{\frac{n+2}{n-2}}$ is nondecreasing on $[\underline u_i,\overline u_i]$. Starting with $u_i^{(0)}:=\underline u_i$, solve successively
\[
  (-c_n\Delta_{h_i}+A_i)u_i^{(m+1)}=(A_i-\Scal_{h_i})u_i^{(m)}+\kappa(u_i^{(m)})^{\frac{n+2}{n-2}}.
\]
The maximum principle and the strict barriers give $\underline u_i\leq u_i^{(m)}\leq u_i^{(m+1)}\leq\overline u_i$. Elliptic compactness gives a smooth limit $u_i$ satisfying $F_i(u_i)=0$ and $\underline u_i\leq u_i\leq\overline u_i$.

If $M$ is open, let $\Omega_1\Subset\Omega_2\Subset\cdots\Subset M$ be a smooth exhaustion. For fixed $i$ and $j$, choose $A_{i,j}>0$ with the same monotonicity property on $\overline\Omega_j$ and perform the preceding iteration with boundary value $1$. Since $\underline u_i<1<\overline u_i$, this gives a smooth solution
\[
  F_i(u_{i,j})=0\quad\text{in }\Omega_j,\qquad u_{i,j}=1\quad\text{on }\partial\Omega_j,\qquad \underline u_i\leq u_{i,j}\leq\overline u_i.
\]
For fixed $i$, interior elliptic estimates and a diagonal subsequence give a smooth global solution $u_i$ of $F_i(u_i)=0$ satisfying the same bounds.

In either case, \eqref{eq:sub-super-def} gives
\[
  \sup_M|u_i-1|\le\max\{1-\underline u_i,\overline u_i-1\}\longrightarrow0.
\]
Lemma~\ref{lem:conformal-regularity} gives $u_i\to1$ in $C^1_{\mathrm{loc}}$. The conformal formula \eqref{eq:conformal-formula} gives $\Scal_{\widehat g_i}=\kappa$, while $\widehat g_i\to g_0$ in $C^0_{\mathrm{loc}}$ and $\{\widehat g_i\}$ is bounded in $W^{1,\infty}_{\mathrm{loc}}$. Lemma~\ref{lem:holder-compactness} proves the required weak-* convergence.

If, in addition, $g_0$ is complete, then,
\[
  \left(1+\frac{5}{2\kappa i}\right)\left(1-\frac1i\right)g_0
  \leq\widehat g_i\leq
  \left(1-\frac{5}{2\kappa i}\right)\left(1+\frac1i\right)g_0.
\]
Thus $\widehat g_i$ is bilipschitz to $g_0$ and hence complete.
\end{proof}
\section{Proof of Theorem~\ref{thm:main} for closed manifolds when \texorpdfstring{$\kappa>0$}{kappa > 0}}\label{sec:positive-closed}

When $\kappa>0$, only the endpoint of Proposition~\ref{prop:global-path} is needed. A small rescaling avoids the kernel of the linearized conformal equation, after which a contraction argument gives the conformal correction.

\begin{lemma}\label{lem:positive-correction}
Let $g$ be a smooth metric on $M$ satisfying
\begin{equation}\label{eq:positive-nonresonance}
  \frac{\kappa}{n-1}\notin\Spec(-\Delta_g).
\end{equation}
Suppose that $h_i$ are smooth metrics satisfying
\[
  h_i\longrightarrow g\quad\text{in }C^0(M),\qquad
  \|\Scal_{h_i}-\kappa\|_{C^0}\longrightarrow0.
\]
Then, for all sufficiently large $i$, there exist smooth functions $u_i>0$ satisfying
\begin{equation}\label{eq:positive-yamabe-equation}
  L_{h_i}u_i=\kappa u_i^{\frac{n+2}{n-2}}
\end{equation}
and $u_i\to1$ uniformly. Consequently, $\widehat g_i:=u_i^{4/(n-2)}h_i$ satisfies $\Scal_{\widehat g_i}=\kappa$ and converges uniformly to $g$.
\end{lemma}

\begin{proof}
Set $e_i:=\Scal_{h_i}-\kappa$ and write $u_i=1+v_i$. Define
\[
  Q(v):=(1+v)^{\frac{n+2}{n-2}}-1-\frac{n+2}{n-2}v.
\]
Then \eqref{eq:positive-yamabe-equation} is equivalent to
\begin{equation}\label{eq:positive-fixed-point}
  A_{h_i}v_i=-e_i-e_iv_i+\kappa Q(v_i),\qquad
  A_h:=-c_n\left(\Delta_h+\frac{\kappa}{n-1}\right).
\end{equation}

The eigenvalues of $A_h$ are
\[
  c_n\left(\lambda_\ell(-\Delta_h)-\frac{\kappa}{n-1}\right),\qquad \ell\in\mathbb N_0.
\]
Uniform convergence of $h_i$ to $g$ gives uniform comparison of the corresponding volume forms and Dirichlet energies. The min--max characterization therefore gives $\lambda_\ell(-\Delta_{h_i})\longrightarrow\lambda_\ell(-\Delta_g)$ for every $\ell$, as well as $\lambda_\ell(-\Delta_{h_i})\geq c\lambda_\ell(-\Delta_g)$ with $c>0$ independent of $i$ and $\ell$. Thus only finitely many eigenvalues can lie in a fixed neighborhood of $\kappa/(n-1)$. Since the spectrum of $-\Delta_g$ is discrete, \eqref{eq:positive-nonresonance} gives $\delta>0$ such that
\[
  \operatorname{dist}\left(\frac{\kappa}{n-1},\Spec(-\Delta_{h_i})\right)\ge\delta
\]
for all sufficiently large $i$. Thus $A_{h_i}$ is invertible and
\begin{equation}\label{eq:positive-L2-inverse}
  \|A_{h_i}^{-1}f\|_{L^2(h_i)}
  \le\frac{1}{c_n\delta}\|f\|_{L^2(h_i)}.
\end{equation}

In local coordinates, after multiplication by $\sqrt{\det h_i}$, the leading coefficient matrix is $\sqrt{\det h_i}\,h_i^{ab}$. The $C^0$-convergence gives uniform bounds and a common ellipticity ratio for these matrices, as well as uniform volume comparison and uniformly bounded zeroth-order coefficients. The local boundedness estimate \cite{GT}*{Theorem~8.17}, applied in a fixed finite atlas, therefore gives
\[
  \|v\|_{C^0}\le C\left(\|v\|_{L^2(h_i)}+\|A_{h_i}v\|_{C^0}\right),
\]
where $C$ is independent of $i$. Applying this estimate to $v=A_{h_i}^{-1}f$ and using \eqref{eq:positive-L2-inverse} gives
\begin{equation}\label{eq:positive-C0-inverse}
  \|A_{h_i}^{-1}f\|_{C^0}\le C\|f\|_{C^0}.
\end{equation}

Set $E_i:=\|e_i\|_{C^0}$ and $\rho_i:=2CE_i$. For $|v|,|w|\le\rho\le1/2$, Taylor expansion gives
\[
  |Q(v)|\le C_n|v|^2,\qquad
  |Q(v)-Q(w)|\le C_n\rho|v-w|.
\]
Let $\BB_i:=\{v\in C^0(M):\|v\|_{C^0}\le\rho_i\}$ and define
\[
  T_i(v):=A_{h_i}^{-1}\bigl(-e_i-e_iv+\kappa Q(v)\bigr).
\]
For $v\in \BB_i$, estimate \eqref{eq:positive-C0-inverse} gives
\[
  \|T_i(v)\|_{C^0}
  \le C\left(E_i+E_i\rho_i+\kappa C_n\rho_i^2\right)
  \le\rho_i
\]
for all sufficiently large $i$. If $v,w\in \BB_i$, then
\[
  \|T_i(v)-T_i(w)\|_{C^0}
  \le C\left(E_i+\kappa C_n\rho_i\right)\|v-w\|_{C^0}
  \le\frac12\|v-w\|_{C^0}
\]
for all sufficiently large $i$. Thus $T_i$ is a contraction of $\BB_i$ into itself. Its fixed point $v_i$ satisfies $\|v_i\|_{C^0}\le\rho_i\longrightarrow0$. In particular, $u_i:=1+v_i>0$ for all sufficiently large $i$. The right-hand side of \eqref{eq:positive-fixed-point} lies in $L^p$ for every finite $p$, so elliptic regularity gives $v_i\in W^{2,p}$. Taking $p>n$ and then applying Schauder estimates iteratively gives $u_i\in C^\infty(M)$. By formula \eqref{eq:conformal-formula}, $\Scal_{\widehat g_i}=u_i^{-\frac{n+2}{n-2}}L_{h_i}u_i=\kappa$.
Finally, $u_i\to1$ and $h_i\to g$ uniformly, so $\widehat g_i\to g$ uniformly.
\end{proof}

\begin{proof}[Proof of Theorem~\ref{thm:main} for closed $M$ and $\kappa>0$] 
Fix $g_0\in\MM_{\ge\kappa}(M)$. For $t\in(0,1)$, set $g_t:=tg_0$. Then
\begin{equation}\label{eq:positive-scaling}
  \Scal_{g_t}=t^{-1}\Scal_{g_0}\ge t^{-1}\kappa>\kappa,\qquad
  \Spec(-\Delta_{g_t})=t^{-1}\Spec(-\Delta_{g_0}).
\end{equation}
Consequently,
\begin{equation}\label{eq:positive-resonant-scales}
  \frac{\kappa}{n-1}\in\Spec(-\Delta_{g_t})
  \quad\Longleftrightarrow\quad
  t=\frac{(n-1)\lambda_\ell(-\Delta_{g_0})}{\kappa}
\end{equation}
for some $\ell\in\mathbb N_0$. Only eigenvalues satisfying $\lambda_\ell(-\Delta_{g_0})<\kappa/(n-1)$ give scales in $(0,1)$. There are only finitely many such eigenvalues, so we may choose nonresonant scales $t_m\uparrow1$. 
For each $m,j\geq1$, apply Proposition~\ref{prop:global-path} to $g_m:=t_mg_0$ with decrement $k_m:=\Scal_{g_m}-\kappa>0$ and $\varepsilon=\eta=j^{-1}$, and retain only the endpoint $s=1$. Denoting this endpoint by $h_{m,j}$, we obtain
\begin{equation}\label{eq:positive-rough-approximation}
  \sup_M|h_{m,j}-g_m|_{g_m}<\frac{1}{j},\qquad
  \|\Scal_{h_{m,j}}-\kappa\|_{C^0}<\frac{1}{j}.
\end{equation}
Thus, for each fixed $m$, $h_{m,j}\to g_m$ in $C^0(M)$, $\Scal_{h_{m,j}}\to\kappa$ uniformly. Since $g_m$ are uniformly bounded in $C^1$, and $k_m$ is uniformly bounded in $C^0$. Hence \eqref{eq:global-first-derivative} allows the metrics $h_{m,j}$ to be chosen uniformly bounded in $C^1$. Since $t_m$ is nonresonant, Lemma~\ref{lem:positive-correction} gives functions $u_{m,j}\to1$ uniformly and metrics $\widehat g_{m,j}:=u_{m,j}^{4/(n-2)}h_{m,j}\in \MM_{\kappa}(M)$ satisfying $\widehat g_{m,j}\to g_m$ in $C^0(M)$ as $j\to\infty$. For each fixed $m$, Lemma~\ref{lem:conformal-regularity} gives $u_{m,j}\to1$ in $C^1(M)$. Choose $j(m)$ sufficiently large that
\[
  \|\widehat g_{m,j(m)}-g_m\|_{C^0_{g_0}}<\frac{1}{m},\qquad \|u_{m,j(m)}-1\|_{C^1}<\frac1m,
\]
and set $\widehat g_m:=\widehat g_{m,j(m)}$. Since $g_m=t_mg_0\to g_0$ uniformly, we obtain $\widehat g_m\to g_0$ in $C^0(M)$. The $C^1$-bounds above show that $\{\widehat g_m\}$ is bounded in $W^{1,\infty}(M)$. Lemma~\ref{lem:holder-compactness} completes the proof.
\end{proof}

\begin{remark}
The argument uses the discreteness of the Laplace spectrum to create a spectral gap at $\kappa/(n-1)$ by rescaling. On an open manifold, this value may lie in the essential spectrum, while parabolicization controls only the bottom of the spectrum.
\end{remark}

\end{document}